\documentclass{amsart}
\usepackage{tabu}
\usepackage[margin=95pt]{geometry}
\usepackage{enumitem}
\allowdisplaybreaks

\usepackage{amssymb}
\usepackage{amsmath}
\usepackage{amscd}
\usepackage{amsbsy}
\usepackage{comment}
\usepackage{enumitem}
\usepackage[matrix,arrow]{xy}
\usepackage{hyperref}
\newcommand{\Q}{{\mathbb Q}}
\newcommand{\Z}{{\mathbb Z}}

\def\mod#1{{\ifmmode\text{\rm\ (mod~$#1$)}
\else\discretionary{}{}{\hbox{ }}\rm(mod~$#1$)\fi}}

\begin {document}

\newtheorem{thm}{Theorem}
\newtheorem{lem}{Lemma}[section]
\newtheorem{prop}[lem]{Proposition}

\newtheorem{cor}[lem]{Corollary}

\theoremstyle{definition}

\theoremstyle{remark}

\title[Perfect $2$-codes over arbitrary alphabets]
{Perfect $2$-codes over arbitrary alphabets}

\author[Michael Bennett]{Michael A. Bennett}
\address{Department of Mathematics, University of British Columbia, Vancouver, B.C., V6T 1Z2 Canada}
\email{bennett@math.ubc.ca}

\date{\today}


\begin {abstract}
The classification of  perfect $e$-codes over an arbitrary alphabet of size $q$ is complete for $e > 2$. In the case of non prime power $q$, it is conjectured that no perfect $2$-codes exist. We confirm this conjecture in a number of situations, including the case where $q=2^\alpha p^\beta$ with $p$ prime, $\alpha$ and $\beta$ positive integers, and either $\alpha \leq 20$, or $\alpha$ sufficiently large.
\end {abstract}
\maketitle

\section{Introduction}

Let $F$ be a finite set (or {\it alphabet}) of cardinality $q \geq 2$, and, for some positive integer $n$, let $V$ be the Cartesian product $F^n$; we call elements of $V$ {\it words}. For any two words ${\bf{x}}$ and ${\bf{y}}$ in V, we define the {\it Hamming distance} $d_F ({\bf{x}}, {\bf{y}})$  to be the number of coordinate places in which ${\bf{x}}$ and ${\bf{y}}$ differ. Given a nonnegative integer $e$ and a word ${\bf{x}} \in V$, we define the sphere of radius $e$, centred at ${\bf{x}}$ to be
$$
S({\bf{x}},e) = \left\{ {\bf{y}} \in V \; : \; d_F ({\bf{x}}, {\bf{y}}) \leq e \right\},
$$
so that the cardinality of $S({\bf{x}},e)$ satisfies
$$
|S({\bf{x}},e)| = \sum_{i=0}^e \binom{n}{i} (q-1)^i.
$$
A subset $\mathcal{C}$ of $V$ is called a {\it code}, and its elements {\it code words}. We call a code $\mathcal{C}$ a {\it perfect $e$-code} if the Hamming spheres $S({\bf{x}},e)$ about the code words ${\bf{x}} \in \mathcal{C}$ form a partition of $V$, i.e. if for every ${\bf{y}} \in V$, there is a unique ${\bf{x}} \in \mathcal{C}$ such that $d_F ({\bf{x}}, {\bf{y}}) \leq e$.

The classification of perfect $e$-codes is almost complete, primarily due to a number of remarkable results from the 1970s.
Restricting our attention to multiple error correcting codes, i.e. to the case $e \geq 2$,  Tiet\"av\"ainen \cite{Tie1} and, independently, Leontiev and Zinoviev \cite{LZ} proved that the only such perfect $e$-codes with $q$ a prime power, are the binary and ternary Golay codes with $(n,q,e)=(23,2,3)$ and $(11,3,2)$. If $q$ is not a prime power, it is conjectured that no perfect $e$-codes exist with alphabet of size $q$. This is known to be true for $e \geq 3$ by work of Best \cite{Best}, Hong \cite{Hong} and Reuvers \cite{Reu} (see also Bannai \cite{Bann}, Laakso \cite{Laak} and  Tiet\"av\"ainen \cite{Tie2}). The case $e=1$ is less well understood, particularly for an arbitrary alphabet; see e.g. Heden \cite{Hed} and Heden and Roos \cite{HR} for what is known. Excellent surveys on perfect codes are available, both classical, due to van Lint \cite{VL1}, \cite{VL2}, and more modern, see e.g. Etzion \cite{Etz} and Heden \cite{Hed}.

To complete the classification of multiple perfect $e$-codes over arbitrary alphabets, it remains then to consider the case $e=2$ with non-prime-power alphabet size $q$.  The main theorem of Reuvers \cite{Reu0} (see also Theorem 4.3.2 of \cite{Reu} and  Theorem 2 of  \cite{CG}) is the following.
\begin{thm}[Reuvers] \label{reu0}
If $q$ is fixed, then the number of perfect $2$-codes with alphabet $q$ is finite.
\end{thm}
\noindent In fact, no such codes are known, and, as noted earlier, it is conjectured that none exist. Results in this direction, however, are rather sparse; we summarize them as follows :
\begin{itemize}
\item There are no perfect $2$-codes on $q$ symbols, if $q=2^\alpha 3^\beta$ for $\alpha$ and $\beta$ positive integers (Bassalygo et al \cite{BZLF}).
\item There are no perfect $2$-codes on $q$ symbols, if 
$q \in \{ 10, 15, 21, 22, 26, 30, 35 \}$ (van Lint \cite{VL1} and Reuvers \cite{Reu}).
\item  There are no perfect $2$-codes on $q$ symbols, if $q=2^\alpha p^\beta$ for $\alpha$ and $\beta$ positive integers, and $p$ prime, if $p \equiv 1 \mod{8}$, or if $p \equiv 5 \mod{8}$ and either $\alpha \geq 2$,  or $\alpha=1$ and $\beta$ even, or if $p=2^t-1$ for $t \geq 3$ an integer (Reuvers \cite{Reu}).
\item There are no perfect $2$-code on $q$ symbols, if 
$q=p_1^\alpha p_2^\beta$, where $p_2>p_1>3$ are primes with $p_2 \equiv 1 \mod{p_1}$,
and $\alpha$ and $\beta$ are positive integers (Reuvers \cite{Reu}).
\item There are no perfect $2$-code on $q$ symbols, for $172$ additional values of $q \leq 600$ (Cazorla Garcia \cite{CG}).
\end{itemize}

All but the last of these date back more than $40$ years; indeed the paucity of advances on these problems led Heden \cite{Hed} to observe in 2010 that ``it is remarkable that since 1982, there [is]  not
one single result in this area to be found, although still the cases $e = 2$ and $e = 1$
with $q$ arbitrary seem to be open''.
The results mentioned above all depend either upon elementary congruential arguments or upon machinery from Diophantine approximation. 
The special case where $q=2p$ for $p$ prime, has been singled out for consideration by van Lint \cite{VL1}, Reuvers \cite{Reu} and Cazorla Garcia \cite{CG}, who all note that their methods do not apply to the general case of arbitrary primes (and indeed $q=2 \cdot 47$ and $q =2 \cdot 83$ are the smallest values of $q$ that remain unresolved). 

Our main results in this paper answer somewhat more general questions. We will in fact show that there are no perfect $2$-codes on $q$ symbols, with, for example,  $q=2 p^\beta$, $q=4 p^\beta$ and $q=8 p^\beta$, for  $\beta$  a positive   integer and $p$ prime.
This is a consequence of

\begin{thm} \label{thm-mainx}
There exist at most finitely many  perfect $2$-codes on $q$ symbols, with $q=2^\alpha p^\beta$, for $\alpha$ and $\beta$  positive integers, and $p$ prime. If such a code exists, necessarily   $\alpha > 20$, $\beta \leq 2519$, $p>10^{10}$ and $p \equiv 3 \mod{8}$. 
\end{thm}

The  bounds here  upon $\alpha$ and $\beta$  can be readily sharpened with effort. 
In the case of general $q$, we strengthen Theorem \ref{reu0} as follows.
\begin{thm} \label{thm-main0}
Let $P_0$ be a fixed positive integer. Then there are at most finitely many perfect $2$-codes on $q$ symbols with greatest prime factor $P(q)$ satisfying $P(q) \leq P_0$.
\end{thm}

For small values of $P_0$, it is relatively straightforward to make this result completely explicit. For instance, we have (as essentially a corollary of  classical work of de Weger \cite{dW})
\begin{thm} \label{thm-main1}
There exist no perfect $2$-codes on $q$ symbols with greatest prime factor $P(q)$ satisfying $P(q) \leq 13$.
\end{thm}

Generalizing both Theorems \ref{thm-mainx} and \ref{thm-main0}, at the cost of (in the latter case)  effectivity, we prove a qualitative but inexplicit  result.
 
 \begin{thm} \label{thm-main!}
Let $P_0$ be a fixed positive integer. Then there are at most finitely many perfect $2$-codes on $q$ symbols, where $q=q_0 p^\alpha$,  $p$ is prime, $\alpha$ is a nonnegative integer, and $q_0$ is a positive integer with greatest prime factor $P(q_0)$ satisfying $P(q_0) \leq P_0$.
\end{thm}

Our proofs depend upon a wide variety of results from Diophantine approximation, effective, ineffective and explicit. By way of example, the proofs of Theorems \ref{thm-main0} and \ref{thm-main1} rely upon bounds for linear forms in complex logarithms and upon lattice-basis reduction techniques, so that, for instance, the implicit bounds upon the parameters for putative codes in Theorem \ref{thm-main0} can be made completely explicit in terms of $P_0$ (which is not the case for the original proof of Theorem \ref{reu0}). Our approach for these results is significantly more
computationally efficient than that employed in \cite{CG}, where the problem of determining perfect $2$-codes over an alphabet of size $q$ is reduced to that of finding integral points $(x,y)$ on  $3^{\omega (q)}$ Mordell curves of the shape $y^2=x^3+k$ (here, $\omega (q)$ denotes the number of distinct prime divisors of $q$), via bounds for linear forms in elliptic logarithms. In this situation, one must compute corresponding Mordell-Weil bases for these curves, where $k \gg q^6$, a very challenging problem for even modest values of $q$. The proof of Theorem \ref{thm-mainx} is rather more intricate, combining bounds for linear forms in  logarithms, the $p$-adic version of Roth's theorem, and a variety of computational arguments. The Schr\"oder-Hipparchus numbers from classical combinatorics even make an unexpected appearance.

\section{Preliminaries and low-hanging fruit} \label{Sec2}

The {\it sphere packing bound} is the necessary condition for the existence of a perfect $e$-code with alphabet of size $q$, that,  for a word ${\bf{x}} \in V$ and code $\mathcal{C}$, we have
$$
 |\mathcal{C} | |S({\bf{x}},e)| =  |\mathcal{C} | \sum_{i=0}^e \binom{n}{i} (q-1)^i 
  =q^n.
 $$
 A further necessary condition for the existence of a perfect $e$-code with alphabet of size $q$, which has, in practice, proven much more useful for classifying such codes, is {\it Lloyd's theorem}, 
 which states that the polynomial
 $$
 \sum_{i=0}^e (-1)^i (q-1)^{e-i} \binom{x-1}{i} \binom{n-x}{e-i}
 $$
 has $e$ distinct integer roots in the interval $[1,n]$. In the context of $q$ which are not prime powers, this was first proved by Lenstra  \cite{Len}. 
 
 Some of the earliest results on perfect codes were deduced by specializing the sphere packing bound to obtain polynomial-exponential equations amenable to Diophantine analysis, such as the classical Ramanujan-Nagell equation, corresponding to $e=q=2$ (see e.g. \cite{Alt}, \cite{Coh}, \cite{Gol} and \cite{SS} for related papers). In most situations, as it transpires, Lloyd's theorem has proven to be of substantially greater utility for classifying perfect codes. Our work combines these two results.
 
For the case $e=2$, we have,
from the sphere packing bound, that
\begin{equation} \label{sphere}
f(n,q)=\frac{n(n-1)(q-1)^2}{2}+n(q-1)+1 \mid q^n.
\end{equation}
Lloyd's theorem, in this situation, implies that the polynomial
\begin{equation} \label{poly}
x^2 - \left( \left( 2 -\frac{2}{q} \right) (n-2) +3 \right)x+\frac{2f(n,q)}{q^2}
\end{equation}
has two distinct integer roots $1 \leq r_1 < r_2 \leq n$, so that
\begin{equation} \label{roots1}
r_1+r_2=\left( 2 -\frac{2}{q} \right) (n-2) +3
\end{equation}
and
\begin{equation} \label{roots2}
r_1 \cdot r_2=\frac{2f(n,q)}{q^2},
\end{equation}
and, in particular, and crucially for our arguments, we have that $r_1$ and $r_2$ are $S$-units, where, from (\ref{sphere}), $S$ is the set of primes dividing $2q$.

Since
$$
q(r_1+r_2)=2(q-1)(n-2)+3q=2n(q-1)+4-q,
$$
it follows that
$$
(q(r_1+r_2)-1)^2+8-(q-3)^2=8f(n,q),
$$
whence
$$
q^2(r_1+r_2)^2-2q(r_1+r_2)+6q-q^2=4 r_1 r_2 q^2,
$$
i.e.
$$
q r_1^2+q r_2^2-2(r_1+r_2)+6-q=2 r_1 r_2 q.
$$
We therefore have
\begin{equation} \label{good}
q (r_2-r_1)^2 = 2(r_1+r_2) +q-6.
\end{equation}

This very simple equation lies at the heart of our arguments; it was first observed by Reuvers  (see Lemma 3.2.7 of \cite{Reu}).
We note that this equation has solutions 
\begin{equation} \label{exxies}
(q,r_1,r_2) = (2,2,4), (10,8,10) \; \mbox{ and } \; \left( \frac{3^t-3}{2}, 3^t-3,3^t \right), \; t \geq 2,
\end{equation}
with every prime divisor of $r_1$ and $r_2$ dividing $2q$. We know of no others.
Notice that, if $(q,r_1,r_2) = \left( \frac{3^t-3}{2}, 3^t-3,3^t \right)$, equation (\ref{roots1}) implies that
$$
2q^2 = (q-1) (n-2),
$$
whence $t=2$, $q=3$ and $n=11$, corresponding to the ternary Golay code. If $(q,r_1,r_2) = (2,2,4)$, from (\ref{roots1}), we have
$n=5$, while $(q,r_1,r_2) = (10,8,10)$ yields $n=31/3$, a contradiction.

Fundamental to our arguments is the fact that the existence of a perfect $2$-code on $q$ symbols implies that the corresponding $S$-units $r_1$ and $r_2$ (where $S$ is the set of primes dividing $2q$) are unusually close together. Specifically, we have the following.
\begin{prop} \label{Goop}
If there exists a perfect $2$-code on $q$ symbols, with words of length $n>2$, then the (integer) roots $r_1<r_2$ of the polynomial (\ref{poly}) satisfy
\begin{equation} \label{size}
0 < r_2-r_1 < \left( \frac{6r_1}{q} \right)^{1/2}.
\end{equation}
\end{prop}

\begin{proof}
Since $n > 2$,  from (\ref{roots1}),  $n \equiv 2 \mod{q/2}$ if $q$ is even, and  $n \equiv 2 \mod{q}$ if $q$ is odd. Further, from (\ref{roots1}),
$$
r_1+r_2 > \left( 2 -\frac{2}{q} \right) n  -1 
$$
and hence, from $r_2 \leq n$, 
\begin{equation} \label{r1low}
r_1 > \left( 1 -\frac{2}{q} \right) n  -1 > 2,
\end{equation}
where we have used only that $q \geq 6$ and $n \geq 5$. Suppose then that (\ref{size}) fails to hold, so that we may write $r_2=r_1+\frac{\kappa}{\sqrt{q}} \,  r_1^{1/2}$, with $\kappa \geq \sqrt{6}$. Then (\ref{good}) implies that
$$
\kappa^2  \, r_1 = 4 r_1+ \frac{2 \kappa}{\sqrt{q}} r_1^{1/2} +q-6,
$$
whence
\begin{equation} \label{what}
\frac{2}{\sqrt{6}} \leq \frac{\kappa^2-4}{\kappa} = \frac{2}{\sqrt{q r_1}} + \frac{q-6}{\kappa r_1}.
\end{equation}
If $q$ is even, $n \equiv 2 \mod{q/2}$ and $n > 2$ imply that  $n \geq 2+q/2$. Thus
$$
r_1 > \left( 1 -\frac{2}{q} \right) n  -1 \geq 
\left( 1 -\frac{2}{q} \right) (2+q/2)-1 =q/2-4/q
$$
and so $r_1 \geq q/2$. If $q$ is odd, $n \equiv 2 \mod{q}$, whence
we have $n \geq 2+q$ and so 
$$
r_1 > \left( 1 -\frac{2}{q} \right) n  -1 \geq 
\left( 1 -\frac{2}{q} \right) (2+q)-1 =q-4/q-1,
$$
and again, crudely, $r_1 \geq q/2$. Thus, from (\ref{what}) and $\kappa \geq 6$, 
$$
\frac{2}{\sqrt{6}}  
< \frac{2\sqrt{2}}{q} + \frac{2}{\sqrt{6}}  -\frac{12}{\sqrt{6} q},
$$
a contradiction.
\end{proof}

The proofs of Theorems \ref{thm-main0} and \ref{thm-main1} are now almost immediate. To prove Theorem  \ref{thm-main0}, note that we may assume $q \geq 10$, so that, from the proof of Proposition \ref{Goop}, $r_1 \geq q/2 > 4$. Theorem 1 of Tijdeman \cite{Tij} thus implies that
$$
r_2-r_1 > \frac{r_1}{\log^\kappa (r_1)},
$$
for an effectively computable positive constant $\kappa$ that depends only on the greatest prime divisor of $q$, $P(q)$. Combining this with (\ref{size}), we find that
$$
r_1 \, q < 6 \, \log^{2\kappa} (r_1),
$$
whence $r_1$ (and hence from (\ref{size}), $r_2$) are effectively bounded in terms of $P(q)$, proving Theorem  \ref{thm-main0}. Theorem \ref{thm-main1} follows upon combining Proposition \ref{Goop} with Theorem 4.6 of de Weger \cite{dW}, an explicit version of a special case of Theorem 1 of \cite{Tij} :
\begin{thm}[de Weger] \label{thm-dW}
There are precisely $605$ pairs of coprime positive integers $(x,y)$ with $P(x) \leq 13$, $P(y) \leq 13$ and
$$
0 < x-y < y^{1/2}.
$$
If, for a positive integer $m$, we write $\nu_p (m)$ for the largest nonnegative integer $k$ for which $p^k \mid m$, then $571$ of these solutions satisfy
$$
\nu_2(xy) \leq 19, \; \nu_3(xy) \leq 12, \; \nu_5(xy) \leq 8, \; \nu_7(xy) \leq 7, \; \nu_{11}(xy) \leq 5 \; \mbox{ and } \; \nu_{13}(xy) \leq 5. 
$$
The remaining $34$ solutions each satisfy all but one of these inequalities for $\nu_p(xy)$, except for
$$
\begin{array}{ll}
(x,y) \; \in & \left\{ (5767168,5764801), (78125000,78121827),  
(781258401,78125000), \right. \\
& \left.  (4882812500,4882786447), 
(13841287201,13841203200),  \right. \\
& \left. (26103515625,26103383072), (26736398612,26736328125), \right. \\
& \left. (13051691536000,13051688172831) \right\}.
\end{array}
$$
In all cases, we have
$$
\nu_2(xy) \leq 26, \; \nu_3(xy) \leq 19, \; \nu_5(xy) \leq 13, \; \nu_7(xy) \leq 13, \; \nu_{11}(xy) \leq 7 \; \mbox{ and } \; \nu_{13}(xy) \leq 8. 
$$
\end{thm}
It is a short calculation to list the $605$ pairs $(x,y)$; the largest corresponds to the identity
\begin{equation} \label{EX}
3^{10} 7^5 13^4 - 2^{10} 5^6 11^6 = 4104623.
\end{equation}
If there is a perfect $2$-code on $q$ symbols with $P(q) \leq 13$, it follows from (\ref{size}) and $q \geq 6$ that there exists a positive integer $d$ such that $r_2=dx$ and $r_1=dy$, for one of these $605$ pairs $(x,y)$. From (\ref{size}), 
$$
d (x-y) < \left( \frac{6dy}{q} \right)^{1/2}
$$
and so
\begin{equation} \label{plip}
q d^{1/2} < 6^{1/2} \left( \frac{y^{1/2}}{x-y} \right).
\end{equation}
 Since
 \begin{equation} \label{plop}
  \prod_{p \mid xy, p > 2} p \mid q,
 \end{equation}
 this provides a very strong upper bound upon $d$. Further, from equation (\ref{good}), if $p > 3$ is a prime dividing $q$, then
 $$
 r_1+r_2 \equiv 3 \mod{p},
 $$
 whereby $P(d) = P(\gcd(r_1,r_2)) \leq 3$. By way of example, for 
 $$
 (x,y) = \left( 3^{10} 7^5 13^4, 2^{10} 5^6 11^6 \right)
 $$
 as in (\ref{EX}), necessarily $q \geq 3 \cdot 5 \cdot 7 \cdot 11 \cdot 13$ and so, from (\ref{plip}), 
 $$
 d < \frac{6 \cdot 2^{10} 5^6 11^6}{\left( 3 \cdot 5 \cdot 7 \cdot 11 \cdot 13 \cdot 4104623 \right)^2} < 1,
 $$
 a contradiction. In fact, for  our $605$ pairs, inequalities (\ref{plip}) and (\ref{plop}) contradict $d \geq 1$, except for $(x,y)=(3,2), (243,242)$ and $(2401,2400)$, where we find that necessarily $d=1$,  
$(x,y)=(4,3), (81,80)$ and $(4375,4374)$, where we have $d \leq 2$, $(x,y)=(128,125)$ with $d \leq 3$, and
$(x,y)=(9,8)$, where  $d \leq 5$. 

 In each of these cases, for values of $d$ with $P(d) \leq 3$, we substitute $r_2=dx$ and $r_1=dy$ into (\ref{good}) and observe that
 $$
 2d (x+y) \equiv 6 \mod{q}.
 $$
 This contradicts {\ref{plop}) unless $(x,y,d)=(9,8,3)$ or $(x,y,d)=(128,125,1)$, in which case
 (\ref{good}) yields $q=12$ and $q=125/2$, respectively,  contradicting (\ref{roots1}). This completes the proof of Theorem \ref{thm-main1}.

\section{$q=2^\alpha p^\beta$ : the proof of Theorem \ref{thm-mainx}} \label{Sec3}

In this section, we will prove Theorem \ref{thm-mainx}. As previously noted, this turns out to be significantly more involved than the  proofs of Theorems  \ref{thm-main0} and   \ref{thm-main1}.
Let us suppose, for the remainder of this section, that $q=2^\alpha p^\beta$ with $p \geq 5$ prime (the case $p=3$ is treated in \cite{BZLF}),  and that $\alpha$ and $\beta$  are positive integers. Then from (\ref{roots1}) and (\ref{roots2}) necessarily $n \equiv 2 \mod{p}$ and $2p^2 \mid f(n,q)$.
 From (\ref{roots1}), 
 $$
 r_1+r_2 \equiv 3 \mod{p}
 $$
 and therefore we may write
 \begin{equation} \label{general}
r_i=2^\kappa \; \mbox{ and } \; r_{3-i} = 2^\omega p^\gamma, \; i \in \{ 1, 2 \},
\end{equation}
 where $\kappa, \omega$ and $\gamma$  are nonnegative integers. As noted earlier, we may assume that $q \geq 94$, so that, from (\ref{size}) and a short calculation, we can suppose that $\kappa \geq 5$, $\kappa > \omega$ and $\gamma \geq 1$.
 
We begin with some elementary remarks. Note first  that (\ref{good}) and (\ref{general}) combine to imply the equation
 \begin{equation} \label{spidey}
 2^{\alpha-1} p^\beta (2^\kappa-2^\omega p^\gamma)^2 = 2^\kappa+2^\omega p^\gamma +2^{\alpha-1} p^\beta-3,
 \end{equation}
 whereby
 \begin{equation} \label{spidey2}
 \min \{ \alpha-1,\omega \} = 0.
  \end{equation}
 Modulo $16$, we have one of
\begin{itemize}
\item $\omega=0$, $p \equiv 3 \mod{8}$ and $\gamma$ is odd,
\item $\omega=1$, $\alpha=1$, $p\not\equiv 1 \mod{8}$, and, if $\beta$ is even, then $4 \mid \beta$, $p \equiv 3 \mod{8}$ and $\gamma$ is odd,
\item $\omega=2$, $\alpha=1$, $p \equiv -1 \mod{8}$ and $\beta$ is odd,
\item $\omega \geq 3$, $\alpha=1$, $p \equiv 3 \mod{8}$ and $\beta$ is odd.
\end{itemize}

In the next few subsections, we will bound $\gamma$ and $\beta$, so that equation (\ref{spidey}) becomes a polynomial-exponential equation (polynomial in $p$), which we would typically approach via Schmidt's Subspace theorem. We will mostly avoid this in order to keep our proofs as effective as possible, though we will, at the very end of this section, appeal to Ridout's $p$-adic version of Roth's theorem, itself a special case of the Subspace theorem.

 \subsection{An upper bound for $\gamma$}
 
 We begin by deducing an upper bound upon the exponent $\gamma$ in (\ref{general}) : 
 \begin{prop} \label{prop-upper}
 If there exists a perfect $2$-code on $q$ symbols, with $q=2^\alpha p^\beta$, for $\alpha$ and $\beta$  positive integers, and $p$ prime, for which the (integer) roots $r_1<r_2$ of the corresponding polynomial (\ref{poly}) are as in (\ref{general}), then 
 $\gamma \leq 5039$.
 \end{prop}
 
 We can sharpen this somewhat through arguing more carefully, but it is adequate for our purposes.
 To obtain this result, we first observe that 
 inequality (\ref{size})  implies  either
 $$
0 <  2^{\kappa-\omega} - p^\gamma  < 2^{(1-\alpha-\omega)/2} \cdot 3^{1/2} \cdot p^{(\gamma-\beta)/2}
 $$
 or
 $$
 0< p^\gamma - 2^{\kappa-\omega}  < 
 2^{(\kappa+1-\alpha-2\omega)/2} \cdot 3^{1/2} \cdot p^{-\beta/2}.
 $$
Writing
$$
\Lambda = \left| (\kappa-\omega) \log 2 - \gamma \log p \right|,
$$
 it follows that we have
 \begin{equation} \label{up1}
 \Lambda = (\kappa-\omega) \log 2 - \gamma \log p < 2^{(1-\alpha-\omega)/2} \cdot 3^{1/2} \cdot p^{(-\gamma-\beta)/2}
 \end{equation}
 or
  \begin{equation} \label{up2}
 \Lambda =  \gamma \log p -  (\kappa-\omega) \log 2< 2^{(1-\kappa-\alpha)/2} \cdot 3^{1/2} \cdot p^{-\beta/2},
 \end{equation}
 respectively. We thus have, from (\ref{size}) and $q \geq 94$, that
 \begin{equation} \label{up3}
 \Lambda < 1.1 \cdot \left( \frac{6}{2^\omega q} \right)^{1/2} p^{-\gamma/2},
  \end{equation}
 in either case.
 
To deduce a lower bound upon $\Lambda$, we use a special case of Corollary 2 of Laurent \cite{Laurent2008}:

\begin{cor}[Laurent]\label{cor:Laurent}
Let $\alpha_1$ and $\alpha_2$ be multiplicatively independent real algebraic numbers $>1$ and $b_1,b_2 \in \Z$ not both zero.
Set $D = [\Q(\alpha_1, \alpha_2): \Q]$.
Then f we have
\[
	\log | b_2 \log \alpha_2 - b_1 \log \alpha_1|
	\geq -25.2 \cdot D^4 \left( \max \{ \log b^\prime + 0.38, 10/D, 1 \} \right)^2 \log A_1 \log A_2,
\]
where $\log A_i \geq \max \{ h(\alpha_i), |\log \alpha_i|/D, 1/D \}$ and
\[
	b^\prime = \frac{|b_1|}{D \log A_2} + \frac{|b_2|}{D \log A_1}.
\]
\end{cor}

We apply this with
$$
\alpha_1=2, \; \alpha_2=p, \; b_1=\kappa-\omega, \; b_2=\gamma, \; D=1,
$$
so that we may choose 
$$
A_1=e, \; A_2=p, 
$$
whence, from (\ref{size}), 
$$
b^\prime = \frac{\kappa}{\log p} + \gamma< 2.5 \gamma
$$
and so
\begin{equation} \label{LFL-lower}
\log \Lambda \geq -25.2 \, \left( \max \{ \log(2.5 \gamma) + 0.38, 10 \} \right)^2 \log p.
\end{equation}	 
Combining this with (\ref{up3}), we thus have
\begin{equation} \label{fritz}
\gamma < 50.4 \, \left( \max \{ \log(2.5 \gamma) + 0.38, 10 \} \right)^2 - \frac{\log (2^\omega q/6)}{\log p}.
\end{equation}
If $\gamma \leq 6024$, $\max \{ \log(2.5 \,\gamma) + 0.38, 10 \}  = 10$ and so $\gamma \leq 5039$. If $\gamma \geq 6025$, then 
$$
\max \{ \log(2.5 \,\gamma) + 0.38, 10 \}  = \log(2.5 \,\gamma) + 0.38, 
$$
and we have
 $$
\gamma < 50.4 \, \left( \log(2.5 \,\gamma) + 0.38 \} \right)^2,
$$
contradicting $\gamma \geq 6025$. In all cases, we thus may conclude that $\gamma \leq 5039$.

 \subsection{A lower bound for $\gamma$}
 In the other direction, we have
 \begin{prop} \label{prop-lower}
 If there exists a perfect $2$-code on $q$ symbols, with $q=2^\alpha p^\beta$, for $\alpha$ and $\beta$  positive integers, and $p$ prime, for which the (integer) roots $r_1<r_2$ of the corresponding polynomial (\ref{poly}) are as in (\ref{general}), then 
 $\gamma \geq \max \{ 2 \beta+1, 4 \}$. In particular, from Proposition \ref{prop-upper}, necessarily $\beta \leq 2519$.
 \end{prop}
 We begin by noting that, from Proposition \ref{Goop}, we have
$$
2^\kappa <2^\omega p^\gamma + \left( \frac{3 \cdot 2^{\omega-\alpha+1} p^\gamma}{p^\beta} \right)^{1/2},
$$
whence (\ref{spidey}) implies
$$
2^{\omega+1}  p^\gamma + \left( \frac{3 \cdot 2^{\omega-\alpha+1} p^\gamma}{p^\beta} \right)^{1/2} +2^{\alpha-1} p^\beta > 2^{2 \omega+\alpha-1} p^\beta (2^{\kappa-\omega}- p^\gamma)^2,
$$
so that 
\begin{equation} \label{kook}
2^{2-\omega-\alpha} p^{\gamma-\beta}+ 3^{1/2} 2^{-3\omega/2-3\alpha/2+3/2} p^{(\gamma-3\beta)/2}
+2^{-2\omega} > (2^{\kappa-\omega}- p^\gamma)^2.
\end{equation}

 Since the right-hand-side here is a positive integer (and at least $9$ if $\gamma \geq 2$), if $\gamma \leq \beta$, we therefore have a contradiction, unless, from (\ref{spidey2}), we have 
$\gamma=\beta =1$ and either $\omega=0$ or $\omega=\alpha=1$. In these cases,
since $2^\kappa \equiv 3 \mod{p}$ implies that we may write 
$2^\kappa = 3 +m  p$, for some positive integer $m$, we have either
\begin{equation} \label{bean1}
2^{\alpha-1} (3+(m-1)p)^2 = m+1 +2^{\alpha-1},
\end{equation}
or
\begin{equation} \label{bean2}
(3+(m-2) p)^2 = m+3,
\end{equation}
respectively.
In case (\ref{bean1}), if $m=1$, we obtain a contradiction modulo $8$, while $m \geq 2$ and (\ref{size}) imply that
$$
\left( \frac{3 }{2^{\alpha-1}} \right)^{1/2} > |2^\kappa-p| = (m-1)p+3 \geq p+3,
$$
an immediate contradiction. In case (\ref{bean2}), if $m=1$, we derive a contradiction since $p >3$, unless $p=5$. In this last situation, (\ref{good}) yields
$$
5 (2^{\kappa-1}-5)^2 = 2^{\kappa-2}+3,
$$
whence we find that $\kappa=3$. This corresponds to $r_1=8, r_2=10$, contradicting (\ref{roots1}). If we have (\ref{bean2}) and $m \geq 2$, then $m=2$ implies a contradiction modulo $8$, while $m \geq 3$ and (\ref{size}) yield 
$$
6^{1/2} > |2^\kappa-2p| = (m-2)p+3 \geq p+3,
$$
once more a contradiction.

We therefore have $\gamma > \beta$, so that $2^\kappa \equiv 3 \mod{p^\beta}$. Then we can write
$$
2^\kappa - 2^\omega p^\gamma = 3+m p^\beta, 
$$
for $m$ an integer. If $m=0$, then $\omega=0$ and so
$$
2^{\alpha+2} p^\beta  = 2^{\omega+1} p^\gamma,
$$
contradicting $\beta < \gamma$. It follows that 
$$
\left| 2^\kappa - 2^\omega p^\gamma \right| \geq p^\beta-3,
$$
while (\ref{size}) implies that
$$
\left| 2^\kappa - 2^\omega p^\gamma \right| < \left(
3 \cdot 2^{\omega-\alpha+1} p^{\gamma-\beta}
 \right)^{1/2},
$$
whence
$$
3 \cdot 2^{\omega-\alpha+1} p^{\gamma-\beta} > p^{2\beta}-6 \cdot p^\beta+9.
$$
If $\omega=0$, we thus have
$$
3 \cdot p^{\gamma-\beta} > p^{2\beta}-6 \cdot p^\beta+9,
$$
so that, if $\gamma \leq 2\beta-1$,
$$
3 \cdot p^{-\beta-1} > 1-6 \cdot p^{-\beta},
$$
contradicting $p \geq 5$.
If, on the other hand, $\omega \geq 1$, then $\alpha=1$ and we have $p^\beta \equiv 3 \mod{2^\omega}$, so that $2^\omega \leq p^\beta-3$, and so
$$
3 \cdot p^\gamma - 9 \cdot  p^{\gamma-\beta} > p^{2\beta}-6 \cdot p^\beta+9.
$$
If $\gamma \leq 2\beta-1$, it follows that
$$
3 \cdot p^{-1} - 9 \cdot  p^{-\beta-1} > 1-6 \cdot p^{-\beta}+9 \cdot p^{-2\beta},
$$
a contradiction unless $\beta=1$, which itself contradicts $\gamma > \beta$ and $\gamma \leq 2\beta-1$.

In all cases, we thus have $\gamma \geq 2 \beta$. To show that $\gamma > 2$, we could appeal to Corollary 1.7 of \cite{BaBe}, but we can argue more generally as follows. If we suppose that $\gamma=2\beta$, then, from (\ref{spidey}), necessarily $\omega=\alpha=1$, $p \equiv 5 \mod{8}$, $\beta$ is odd, and $2^\kappa \equiv 3 \mod{p^\beta}$, say $2^\kappa=3+k_0p^\beta$. Equation (\ref{spidey}) thus becomes
$$
(3+k_0p^\beta-2 p^{2\beta})^2 = k_0+2 p^{\beta} +1,
$$
so that 
$$
4p^{4\beta} -4k_0p^{3\beta} + (k_0^2-12) p^{2\beta}+(6k_0-2)p^\beta +(8-k_0)= 0.
$$
We therefore have $k_0 \equiv 8 \mod{p^\beta}$, say $k_0=8+k_1 p^\beta$, whereby 
$$
2^\kappa \equiv 3 + 8 p^\beta \mod{p^{2\beta}}.
$$
It follows that there exists an integer $m$ such that 
$$
2^\kappa - 2 p^{2\beta} = 3 + 8 p^\beta + m p^{2\beta}.
$$
From parity considerations, $m$ is odd. If $m \geq 1$, we have
$$
\left( 6 p^\beta \right)^{1/2} > \left| 2^\kappa - 2 p^{2\beta} \right| = 
3 + 8 p^\beta + m p^{2\beta} > p^{2\beta},
$$
contradicting $p \geq 5$.
We thus have $m \leq -1$, so that
$$
\left( 6 p^\beta \right)^{1/2} > \left| 2^\kappa - 2 p^{2\beta} \right| = 
\left| 3 + 8 p^\beta + m p^{2\beta} \right|  \geq p^{2\beta}-8 p^\beta-3,
$$
a contradiction unless $\beta=1$ and $p \in \{ 5, 7 \}$. In these cases, there exists no $\kappa$ for which
$$
\left( 6 p^\beta \right)^{1/2} > \left| 2^\kappa - 2 p^{2\beta} \right|.
$$

We may thus conclude that $\gamma \geq 2 \beta+1$. In particular, we have $\gamma \geq 3$.
 We now show that $\gamma \geq 4$. 
If $\gamma=3$, again $\beta=1$ and Proposition \ref{Goop} yields the inequality
$$
|2^{\kappa-\omega}-p^3| <  3^{1/2} p.
$$
Corollary 2.2 of Bauer and Bennett \cite{BaBe2} gives that either
$$
|2^{\kappa-\omega}-p^3| \geq p^{4/3},
$$
or $p =5$ and $\kappa-\omega=7$. Since $p \geq 5$, it follows in all cases that $p =5$ and $\kappa-\omega=7$. We thus have $\omega=1$, so that $\kappa=8$, contradicting $2^\kappa \equiv 3 \mod{5}$. This completes the proof of Proposition \ref{prop-lower}.

 \subsection{A lower bound for $p$}
 \begin{prop} \label{prop-lowerp}
 If there exists a perfect $2$-code on $q$ symbols, with $q=2^\alpha p^\beta$, for $\alpha$ and $\beta$  positive integers, and $p$ prime, then $p > 10^{10}$.
 \end{prop}
 
From (\ref{up3}) and $q \geq 94$ (see \cite{CG}), we have
\begin{equation} \label{snoopy!!}
\left|  \frac{\log p}{\log 2}- \frac{\kappa-\omega}{\gamma} \right| < \frac{1.1 \left( \frac{6}{2^\omega q} \right)^{1/2} }{\gamma \, p^{\gamma/2} \log 2}< \frac{1}{2 \gamma p^\gamma}.
\end{equation}
Since $p^{\gamma/2} \geq \gamma$, a classical result of Legendre implies that $(\kappa-\omega)/\gamma= p_i/q_i$, where $p_i/q_i$ is a convergent in the infinite simple continued fraction expansion to $\frac{\log p}{\log 2}$ (with corresponding partial quotients $a_i$). 
We have
\begin{equation} \label{snoopy2}
\left|  \frac{\log p}{\log 2}- \frac{p_i}{q_i} \right| > \frac{1}{(a_{i+1}+2) q_i^2}
\end{equation}
and so, we require
\begin{equation} \label{snoopy3}
a_{i+1} > 2 \, p^{q_i/2}  q_i^{-1}-2,
\end{equation}
where, additionally, $q_i \leq \gamma \leq 5039$. Note that since the Fibonacci number $F_{20}=6765>5039$, we necessarily have $i \leq 19$. Notice further that $q_i=1$ implies that $\kappa-\omega = p_i \gamma$ and so 
$$
\left| 2^{\kappa-\omega} - p^\gamma \right| > \gamma \left( \min \{ p, 2^{p_i} \} \right)^{\gamma-1}. 
$$
Proposition \ref{Goop} thus implies that either 
$$
3^{1/2} 2^{(1-\alpha-\omega)/2} p^{(\gamma-\beta)/2} > \gamma p^{\gamma-1}
$$
or that
$$
3^{1/2} 2^{(1+\kappa-\alpha-2\omega)/2} p^{-\beta/2} >  \gamma 2^{\kappa-\omega-(\kappa-\omega)/\gamma}.
$$
In each case, we contradict $\gamma \geq 4$ and $p \geq 5$.
 We may thus conclude that $q_i > 1$ so that, in particular, $1 \leq i \leq 19$.

We will use this argument to deduce that there are no  perfect $2$-code on $q$ symbols, with $q=2^\alpha p^\beta$, for $\alpha$ and $\beta$  positive integers, and $p \leq 10^{10}$ prime. For each such small prime $p$ (there are $455052509$ primes $5 \leq p < 10^{10}$; we may assume that $p \not\equiv 1 \mod{8}$),
we compute the first, say, $20$ terms in the infinite simple continued fraction expansion to $\frac{\log p}{\log 2}$ and check to see whether any of the convergents $p_i/q_i$ and partial quotients $a_i$ satisfy (\ref{snoopy3}), with $2 \leq q_i \leq 5039$.
We find only the examples corresponding to
$$
p=5, \; p_1/q_1=7/3, \; a_2=9
$$
and 
$$
p=181, \; p_1/q_1=15/2, \; a_2=1621.
$$
Since these are both $5 \mod{8}$, we thus have $\alpha=\omega=1$ in each case and, from (\ref{spidey}), that $\beta$ is odd, so that the first inequality in 
(\ref{snoopy!!}) becomes
$$
\left|  \frac{\log p}{\log 2}- \frac{\kappa-\omega}{\gamma} \right| < \frac{1.1 \left( \frac{3}{2 p^\beta} \right)^{1/2} }{\gamma \, p^{\gamma/2} \log 2},
$$
where we have
$$
(p,\gamma, (\kappa-1)/\gamma) = (5,3d_1,7/3) \; \mbox{ or } \; (181,2d_2,15/2), 
$$
for positive integers $d_1$ and $d_2$. A short computation reveals that, in case $p=5$, we have $\beta =1$ and $d_1=1$, while $p=181$ yields an immediate contradiction. Since we have assumed that $q \geq 94$, this completes the proof of Proposition \ref{prop-lowerp}. We note that this computation, while time consuming, is readily parallelized (taking somewhat less than $10$ hours in Maple, using $16$ cores on a 2024 MacBook Pro).

 \subsection{An expansion in powers of $p^\beta$}
  
Crucial to our arguments is the following.
\begin{prop} \label{ghost}
Let $t = \left[ \gamma/\beta \right]$ (so that, from Proposition \ref{prop-lower}, $t \geq 2$).
We have
$$
2^\kappa \equiv  3 + \sum_{i=1}^{t-1} \mathfrak{C}_{i+1} 2^{2+\alpha i} p^{\beta i} \mod{p^{t \beta}},
$$
where $\mathfrak{C}_n$ denotes the Schr\"oder-Hipparchus {\it aka} super-Catalan number defined via
$$
\mathfrak{C}_1=\mathfrak{C}_2=1
$$
and, for $n \geq 3$,
$$
\mathfrak{C}_n = \frac{3(2n-3) \mathfrak{C}_{n-1} -(n-3) \mathfrak{C}_{n-2}}{n}.
$$
\end{prop}

\begin{proof}
 We have the generating function for the Schr\"oder-Hipparchus numbers (see e.g. Stanley \cite{S}) 
$$
S(x)=\sum_{n=1}^\infty  \mathfrak{C}_n x^n = 
\frac{1}{4} \left( 1+x -\sqrt{1-6x+x^2} \right),
$$
so that
$$
2 S(x)^2-(1+x) S(x) +x=0.
$$
Substituting, expanding and equating coefficients of $x$, we find that 
$$
 \mathfrak{C}_{k+2} +\mathfrak{C}_{k+1} = 
2  \sum_{i=0}^{k}   \mathfrak{C}_{i+1} \mathfrak{C}_{k-i+1}
$$
and so
\begin{equation} \label{pump}
 \mathfrak{C}_{k+2} =3 \mathfrak{C}_{k+1} 
+2  \sum_{i=1}^{k-1}   \mathfrak{C}_{i+1} \mathfrak{C}_{k-i+1}.
\end{equation}

We will show that 
\begin{equation} \label{mug}
2^\kappa \equiv  3 + \sum_{i=1}^{k} \mathfrak{C}_{i+1} 2^{2+\alpha i} p^{\beta i} \mod{p^{(k+1)\beta}},
\end{equation}
for all $0 \leq k < t=\left[ \gamma/\beta \right]$, which will complete the proof of Proposition \ref{ghost}. Notice that, from (\ref{spidey}), we have $2^\kappa \equiv 3 \mod{p^\beta}$ and hence (\ref{mug}) with $k=0$. Assume that we have (\ref{spidey}) for some fixed $0 \leq k \leq t-2$. Then 
$$
2^\kappa =  3 + \sum_{i=1}^{k} \mathfrak{C}_{i+1} 2^{2+\alpha i} p^{\beta i} +M p^{(k+1)\beta},
$$
for some integer $M$ and so (\ref{spidey}) yields
$$
2^{\alpha-1} (3+\sum_{i=1}^{k} \mathfrak{C}_{i+1} 2^{2+\alpha i} p^{\beta i} +M p^{(k+1)\beta}-2^\omega p^\gamma)^2 = 2^\omega p^{\gamma-\beta} +2^{\alpha-1} + \sum_{i=1}^{k} \mathfrak{C}_{i+1} 2^{2+\alpha i} p^{\beta (i-1)} +M p^{k\beta}.
$$
From $\gamma \geq t \beta \geq (k+2)\beta$,
$$
2^{\alpha-1} \left( 3+\sum_{i=1}^{k} \mathfrak{C}_{i+1} 2^{2+\alpha i} p^{\beta i} \right)^2 \equiv 2^{\alpha-1} + \sum_{i=1}^{k} \mathfrak{C}_{i+1} 2^{2+\alpha i} p^{\beta (i-1)} +M p^{k\beta}
\mod{p^{(k+1)\beta}}.
$$
Expanding, it follows that 
$$
2^{\alpha-1} \left( \sum_{i=1}^{\left[ \frac{k+1}{2} \right]} \mathfrak{C}_{i+1}^2 2^{4+2\alpha i} p^{2\beta i}  + 6 \sum_{i=1}^{k} \mathfrak{C}_{i+1} 2^{2+\alpha i} p^{\beta i} +  2 \sum_{i=1}^{\left[ \frac{k+1}{2} \right]}  \sum_{j=i+1}^{k-i} \mathfrak{C}_{i+1} \mathfrak{C}_{j+1} 2^{4+\alpha (i+j)} p^{\beta (i+j)}  \right)
$$
is congruent modulo $p^{(k+1)\beta)}$ to
$$
 \sum_{i=2}^{k} \mathfrak{C}_{i+1} 2^{2+\alpha i} p^{\beta (i-1)} +M p^{k\beta},
$$
whereby, appealing to (\ref{pump}),
$$
M \equiv 3 \mathfrak{C}_{k+1} 2^{(k+1)\alpha+2}
+ 2^{(k+1)\alpha+3}  \sum_{i=1}^{k-1}   \mathfrak{C}_{i+1} \mathfrak{C}_{k-i+1}
\equiv \mathfrak{C}_{k+2} 2^{2+(k+1)\alpha}  \mod{p^\beta},
$$
which gives us (\ref{mug}) with $k$ replaced by $k+1$, as desired. 
This completes the proof of Proposition \ref{ghost}.
\end{proof}

\subsection{Applying Proposition \ref{ghost}}
 
 We suppose, from Propositions \ref{prop-lower} and \ref{prop-lowerp}, that $p > 10^{10}$ and 
$$
\max \{ 2 \beta+1, 4 \} \leq \gamma \leq 5039.
$$
Notice further that $\gamma$ even implies from (\ref{spidey}) that $\omega=\alpha=1$, $p \equiv 5 \mod{8}$ and $\beta$ is odd.
 
\subsubsection{The case $\alpha \geq 2$}

If $\alpha \geq 2$, then $\omega = 0$ and $\gamma$ is odd. From $p > 10^{10}$ and (\ref{snoopy!!}), we have 
\begin{equation} \label{magnet}
\kappa \geq \kappa-\omega > 0.99 \frac{\log p}{\log 2} \, \gamma > 33 \gamma,
\end{equation}
while equation (\ref{spidey}) yields the congruence
\begin{equation} \label{fridge}
2^{\alpha-1} p^{2\gamma+\beta} - p^\gamma - 2^{\alpha-1} p^\beta +3 \equiv 0 \mod{2^{\kappa}}.
\end{equation}
We will employ these later to improve our lower bound upon $p$.

To proceed, we note that, from Proposition \ref{Goop}, 
\begin{equation} \label{kane}
\left| 2^\kappa - p^\gamma \right| <  \left( \frac{3 \cdot 2^{1-\alpha} p^\gamma}{p^\beta} \right)^{1/2},
\end{equation}
while Proposition \ref{ghost} implies that
$$
2^\kappa - p^\gamma = 3 + \sum_{i=1}^{t-1} \mathfrak{C}_{i+1} 2^{2+\alpha i} p^{\beta i} +m p^{t \beta},
$$
for some integer $m$. If $m \geq 0$, then 
$$
2^\kappa - p^\gamma > \mathfrak{C}_{t} 2^{2+\alpha (t-1)} p^{\beta (t-1)},
$$
which, in conjunction with (\ref{kane}) yields
$$
\mathfrak{C}_{t} 2^{2+\alpha (t-1)} p^{\beta (t-1)} <
3^{1/2} 2^{(1-\alpha)/2} p^{t \beta/2},
$$
contradicting $t \geq 2$. We thus have $m \leq -1$ and so, from (\ref{kane}), 
\begin{equation} \label{fossil}
3^{1/2} 2^{(1-\alpha)/2} p^{t \beta/2} > p^{t\beta} -3-
\sum_{i=1}^{t-1} \mathfrak{C}_{i+1} 2^{2+\alpha i} p^{\beta i}.
\end{equation} 

Observe that, immediately from $\mathfrak{C}_1=\mathfrak{C}_2=1$ and the recursion for $\mathfrak{C}_n$, we have 
\begin{equation} \label{boop}
\mathfrak{C}_n < 6^n,
\end{equation}
for all $n \geq 1$. This can be improved, since it is relatively straightforward to show that 
$$
\log \mathfrak{C}_n \sim \log (3+2\sqrt{2}) n, 
$$
but is adequate for our purposes. Inequality (\ref{fossil}) thus implies that
\begin{equation} \label{trap}
1 < 3 \cdot p^{-t\beta} + 3^{1/2} 2^{(1-\alpha)/2} p^{-t \beta/2}
+\sum_{i=1}^{t-1} 6^{i+1} 2^{2+\alpha i} p^{\beta (i-t)},
\end{equation}
and hence that
$$
p^{\beta} < 3 \cdot p^{(1-t)\beta} + 3^{1/2} 2^{(1-\alpha)/2} p^{1+t \beta/2}
+24 \cdot \left( 6 \cdot 2^{\alpha} \right)^{t-1} 
\frac{6q}{6q-1}.
$$
From $t \geq 2$, we therefore have
\begin{equation} \label{trappy}
p^\beta < 25 \left( 6 \cdot 2^\alpha \right)^{t-1}.
\end{equation}

If we suppose now that $\alpha \leq 20$, then $\gamma \leq 5039$ and (\ref{trappy}) imply that
$$
p^\beta < 25 \cdot \left( 6 \cdot 2^{20} \right)^{\frac{\gamma}{\beta}-1} < \left( 6 \cdot 2^{20} \right)^{\frac{\gamma}{\beta}} 
$$
and hence,
since $p > 10^{10}$, 
\begin{equation} \label{b-bound}
1 \leq \beta < 0.825 \, \gamma^{1/2}.
\end{equation}

We will proceed as follows. For each odd $\gamma$,  with $5 \leq \gamma \leq 5039$, each $\alpha$ with $2 \leq \alpha \leq 20$, and each $\beta$ satisfying (\ref{b-bound}) (these inequalities correspond to precisely $1942740$ triples $(\gamma,\alpha,\beta)$), we will solve a number of congruences of the shape
\begin{equation} \label{farkle}
2^{\alpha-1} p^{2\gamma+\beta} - p^\gamma - 2^{\alpha-1} p^\beta +3 \equiv 0 \mod{2^{k}},
\end{equation}
which we know to hold, from (\ref{fridge}), for all $k \leq \kappa$. For each triple $(\gamma,\alpha,\beta)$ and suitably large $k$, this will lead to a lower bound for $p$ that contradicts (\ref{trappy}). Specifically, noting that (\ref{magnet}) and $\gamma \geq 5$ together imply that we always have $\kappa >100$, we will first solve the congruence (\ref{farkle}) with $k=100$, and then use quadratic Hensel lifting to extend this to solutions with $k=2^j \cdot 100$, for 
$$
j = 1, 2, \ldots, \left[ \frac{\log (\kappa/100)}{\log 2} \right].
$$

To be precise, we define a polynomial
$$
f(x) = 2^{\alpha-1} x^{2\gamma+\beta} - x^\gamma - 2^{\alpha-1} x^\beta +3,
$$
so that
$$
f^\prime(x) = 2^{\alpha-1} (2\gamma+\beta) x^{2\gamma+\beta-1} - \gamma x^{\gamma-1} - 2^{\alpha-1} \beta x^{\beta-1}
$$
is necessarily, from $\alpha \geq 2$, odd whenever $f$ is even (whereby our Hensel lifts are unique). 
We compute positive integers $r_j$ such that $f(r_j) \equiv 0 \mod{2^{2^j \cdot 100}}$, as follows. Via built-in code from Maple2024, we find the least positive integer $r_0$ such that 
$$
f(r_0) \equiv 0 \mod{2^{100}},
$$
and define $d_0$ to be the smallest positive integer with
$$
d_0 f^\prime (r_0) \equiv 1 \mod{2^{100}}.
$$
We then define iteratively, for $j=1, 2, \ldots$, least positive residues  $r_j$ and $d_j$ with the properties that
$$
f(r_j) \equiv 0 \mod{2^{2^j \cdot 100}},
$$
and 
$$
d_j f^\prime (r_j) \equiv 1 \mod{2^{2^j \cdot100}},
$$
via
$$
r_j \equiv r_{j-1} - f(r_{j-1}) \cdot d_{j-1} \mod{2^{2^j \cdot 100}}
$$
and
$$
d_j \equiv d_{j-1} (2- f^\prime (r_j) d_{j-1}) \mod{2^{2^j \cdot 100}}.
$$
It follows from (\ref{magnet}) and (\ref{fridge}) that necessarily $p \geq r_j$ for every $j$ with $2^j \cdot 100 < 33 \gamma$. Running this computation, we find, for each triple $(\gamma,\alpha,\beta)$, a value of $j$ for which 
$$
2^j \cdot 100 < 33 \gamma \; \mbox{ and } \; r_j > 2^{23 \gamma} > 25 \cdot t \left( 6 \cdot 2^\alpha \right)^{t-1},
$$
contradicting $p \equiv r_j \mod{2^{2^j \cdot 100}}$ and (\ref{trappy}). In fact, in every case, we have $j \leq 11$. By way of example, if we take the extreme case $(\gamma,\alpha,\beta)=(5039,20,1)$, our initial values are
$$
r_0=1017224537123983744313102132587 \; \mbox{ and } \; g_0=864194828523587462947121421225.
$$
Iterating, we compute $r_1, r_2, \ldots, r_{11}$ and find that they satisfy the following.
$$
\begin{array}{cc|cc}
j & \left[ \frac{\log r_j}{\log 2} \right] & j & \left[ \frac{\log r_j}{\log 2} \right] \\ \hline
1 & 199 & 7 & 12798 \\
2 & 399 & 8 & 25596 \\
3 & 798 & 9 & 51199 \\
4 & 1599 & 10 & 102394 \\
5 & 3199 & 11 & 204798. \\
6 & 6398 & & \\ 
\end{array}
$$
Since $204798 > 23 \cdot 5039$, we have $p \geq r_{11} > 2^{23 \gamma}$ and reach our desired contradiction from (\ref{trappy}).

\subsubsection{The case $\alpha =1$} We next treat the case $\alpha=1$. Equation (\ref{spidey})  now becomes
\begin{equation} \label{spidey3}
 p^\beta (2^\kappa-2^\omega p^\gamma)^2 = 2^\kappa+2^\omega p^\gamma + p^\beta-3.
 \end{equation}
From Proposition \ref{ghost}, we have
$$
2^\kappa - 2^\omega p^\gamma = 3 + \sum_{i=1}^{t-1} \mathfrak{C}_{i+1} 2^{2+ i} p^{\beta i} +m p^{t \beta},
$$
for an integer $m$. Again, if $m \geq 0$, then 
$$
2^\kappa - 2^\omega p^\gamma > \mathfrak{C}_{t} 2^{t+1} p^{\beta (t-1)},
$$
which, from  Proposition \ref{Goop}, implies that 
$$
\mathfrak{C}_{t} 2^{t+1} p^{\beta (t-1)} <
3^{1/2} 2^{\omega/2} p^{(\gamma-\beta)/2}
< 3^{1/2} 2^{\omega/2} p^{t \beta/2}. 
$$
If $\omega=0$, this immediately contradicts $t \geq 2$. If $\omega \geq 1$, then necessarily $m$ is odd and we have, 
from Proposition \ref{Goop}, 
$$
3^{1/2} 2^{\omega/2} p^{t \beta/2} > p^{t\beta}.
$$
Since Proposition \ref{Goop} implies that 
$\omega < \kappa$,  (\ref{spidey3}) thus yields $p^\beta \equiv 3 \mod{2^\omega}$, whence $2^\omega < p^\beta$. We therefore have 
$$
3^{1/2}  p^{(t+1)\beta/2} > p^{t\beta},
$$
contradicting $t \geq 2$. We may thus suppose that 
 $m \leq -1$ and so, from (\ref{kane}), 
\begin{equation} \label{fossil2}
3^{1/2} p^{(t+1) \beta/2} > p^{t\beta} -3-
\sum_{i=1}^{t-1} \mathfrak{C}_{i+1} 2^{2+ i} p^{\beta i}.
\end{equation} 
Arguing as previously, it follows from (\ref{boop}), $t \geq 2$ and $p>10^{10}$ that
$$
p^\beta < 2.1 \cdot 12^{t}.
$$
Thus
\begin{equation} \label{trappy2}
p^\beta < 2.1 \cdot 12^{\gamma/\beta}.
\end{equation}
Since $p > 10^{10}$, it follows that 
\begin{equation} \label{b-bound2}
1 \leq \beta < 0.355 \, \gamma^{1/2}.
\end{equation}
In particular, we thus have $\gamma \geq 9$ and, from $\gamma \leq 5039$, $\beta \leq 25$.

For $\beta=25$, the congruence $p^{25} \equiv 3 \mod{2^{36}}$ implies that $p \equiv 40882718819 \mod{2^{36}}$ and so
$p \geq 40882718819$. From (\ref{trappy2}) and $p^\beta \equiv 3 \mod{2^\omega}$, we deduce a contradiction and so $\omega \leq 35$. Arguing in a like fashion for each odd $\beta$, $1 \leq \beta \leq 25$, we obtain  upper bounds for $\omega$ as follows.

$$
\begin{array}{cc|cc}
\beta & \omega \leq \omega_0 & \beta & \omega \leq \omega_0 \\ \hline
1 & \omega_0=18077 & 15 & \omega_0=81 \\
3 &  \omega_0=2013 & 17 & \omega_0=63 \\
5 & \omega_0=729 & 19 & \omega_0=51 \\
7 & \omega_0=370 & 21 &  \omega_0=42 \\
9 & \omega_0=224 & 23 &  \omega_0=38 \\
11 & \omega_0=151 & 25 &  \omega_0=35 \\
13 &  \omega_0=108 & & \\
\end{array}
$$
Again combining  (\ref{trappy2}) and $p^\beta \equiv 3 \mod{2^\omega}$, we have
\begin{equation} \label{trappy3}
2^\omega < 2.1 \cdot \frac{\gamma}{\beta}  \cdot 12^{\frac{\gamma}{\beta}}.
\end{equation}

We are thus left to treat  triples $(\gamma,\beta, \omega)$ satisfying (\ref{b-bound2}), (\ref{trappy3})  and one of
\begin{itemize}
\item $\omega=0$, $\gamma$ odd, $9 \leq \gamma \leq 5039$, or
\item $\omega=1$, $\gamma$ odd, $9 \leq \gamma \leq 5039$, $\beta \not\equiv 2 \mod{4}$, or
\item $\omega=1$, $\gamma$ even, $10 \leq \gamma \leq 5038$, $\beta$ odd, or
\item $\omega \geq 2$,  $9 \leq \gamma \leq 5039$,  $\beta$ odd and $\omega \leq \omega_0(\beta)$.
\end{itemize}
There are precisely $92725$ such triples with $\omega \leq 1$; the great majority of triples correspond to $\omega \geq 2$, namely, $61634419$. For each triple, we will proceed analogously to the situation with $\alpha>1$, noting, from (\ref{spidey3}), the congruence 
\begin{equation} \label{fridge2}
2^{2\omega} p^{2\gamma+\beta} - 2^\omega p^\gamma - p^\beta+3  \equiv 0 \mod{2^\kappa}.
\end{equation}
Once again, we have inequality (\ref{magnet}), and so, since $\gamma \geq 4$, $\kappa > 100$ in all cases. As previously, we solve 
$$
2^{2\omega} r_j^{2\gamma+\beta} - 2^\omega r_j^\gamma - r_j^\beta+3  \equiv 0 \mod{2^{2^j \cdot 100}},
$$
for $j=0, 1, 2, \ldots$, use the inequality $p \geq r_j$ to, via (\ref{magnet}), increase our lower bound upon $\kappa$, so that we may iterate the process, eventually arriving at a lower bound for $p$ that contradicts (\ref{trappy2}). 
A potential  minor additional complication for our Hensel lifting comes from the fact that the polynomial  $f(x)=2^{2\omega} x^{2\gamma+\beta} - 2^\omega x^\gamma - x^\beta+3$  has an even derivative if $\omega=1$ and $4 \mid \beta$, but this causes no problems as 
we have $f^\prime (x) \equiv 2 \mod{4}$ for all $x$, in this case, and, as before, we begin by solving $f(x) \equiv 0 \mod{2^{100}}$. The sheer numbers of triples involved (in excess of $6 \times 10^7$)  makes this a somewhat slow process, though readily parallelized. In every case, we ultimately contradict (\ref{trappy2}).

 \subsection{The proof of Theorem \ref{thm-mainx} concluded}
 
It remains, then, to show that there are finitely many perfect $2$-codes on $q$ symbols, with $q=2^\alpha p^\beta$, $p$ prime and $\alpha > 20$. Unlike our work in case $\alpha \leq 20$, here our conclusion is ineffective. We will appeal to the following result of Ridout \cite{Rid}.
\begin{thm}[Ridout]\label{Riddy}
Let $\alpha$ be a nonzero algebraic number and $P_1, P_2, \ldots, P_s$, $Q_1, Q_2, \ldots, Q_t$  distinct primes. Suppose that $\mu, \nu$ and $c$ are given real numbers with
$$
0 \leq \mu \leq 1, \; \; 0 \leq \nu \leq 1 \;  \mbox{ and }  \; c > 0,
$$
and restrict our attention to positive integers $p$ and $q$ with the property that the largest divisor $p_0$ of $p$ that is coprime to $P_1P_2 \cdots P_s$, and the largest divisor $q_0$ of $q$  coprime to $Q_1Q_2 \cdots Q_t$ satisfy
$$
p_0 \leq c p^\mu \; \mbox{ and } \; q_0 \leq c q^\nu.
$$
Then, given $\epsilon > 0$, the inequality
$$
\left| \alpha - \frac{p}{q} \right| < \frac{1}{q^{\mu+\nu+\epsilon}}
$$
has at most finitely many solutions in such $p$ and $q$.
\end{thm}

Note that we have
\begin{equation} \label{knife}
\left|  2^\kappa-p^\gamma \right| < 3^{1/2} 2^{(1+\kappa-\alpha)/2 } p^{-\beta/2}
\end{equation}
and write $\kappa/\gamma = \delta+\kappa_0/\gamma$, where $\kappa_0$ is an integer with $0 \leq \kappa_0 < \gamma$, so that,
crudely,
\begin{equation} \label{jump}
\left| 2^{\kappa_0/\gamma} - \frac{p}{2^\delta} \right| < 2^{-\delta \gamma/2}.
\end{equation}
Applying Theorem \ref{Riddy} for each odd $\gamma$ with $5 \leq \gamma \leq 5039$, and each $0 \leq \kappa_0 < \gamma$, setting $\alpha =2^{\kappa_0/\gamma}$, $Q_1=2$, $c=\mu=1$, $\nu=0$,
since $\gamma > 4$, there thus exist at most finitely many $p$ satisfying (\ref{jump}). An application of Theorem \ref{thm-main0}
thus completes the proof of Theorem \ref{thm-mainx}.

\section{Proof of Theorem \ref{thm-main!}}

Suppose that $P_0 \geq 2$ is a given  integer.  Suppose further  that $q=q_0p^\alpha$ for $p$ prime and $q_0$ an integer with $P(q_0) \leq P_0$. Without loss of generality, assume that $p > \max \{ P_0, 1000 \}$. Then the corresponding zeros of the quadratic Lloyd polynomial, $r_1$ and $r_2$, necessarily satisfy
$$
r_i = q_1 \; \mbox{ and } \; r_{3-i} = q_2 p^\gamma, \; \mbox{ for } i \in \{ 1, 2 \},
$$
where $P(q_1) \leq P_0$, $P(q_2) \leq P_0$ and $\gamma$ is a nonnegative integer. If $\gamma=0$, the desired result follows from Theorem \ref{thm-main0}, so  we will assume that $\gamma$ is positive. If we set $d=\gcd (q_1, q_2)$, then, from (\ref{good}), $P(d) \leq 3$.
Suppose that $\ell^\kappa$ is the largest prime power dividing $r_1r_2/d^2$, so that $\kappa \log\ell = \theta \gamma \log p$, where $\theta \geq 1$. From Proposition \ref{Goop}, we have, very crudely,
\begin{equation} \label{drive}
\left| \frac{r_2}{r_{1}} -1 \right|  < \ell^{-\kappa/2}.
\end{equation}
To derive a lower bound upon the left-hand-side of (\ref{drive}), we will appeal to a result of Matveev \cite{Mat}, a bound for linear forms in complex logarithms, as presented in slightly modified form by Bugeaud as Theorem 2.2 of \cite{Bug}.
\begin{thm}[Matveev] \label{Mat} Let $n \geq 1$ be an integer and
suppose that $\alpha_1, \alpha_2, \ldots, \alpha_n$ are algebraic numbers, and that $b_1, b_2, \ldots, b_n$ are integers.
Let $D=\mathbb{Q}(\alpha_1,\alpha_2, \ldots, \alpha_n)$ and $A_1, A_2, \ldots, A_n$ be real numbers satisfying
$$
\log A_j \geq \max \left\{ h(\alpha_j), \frac{|\log \alpha_j|}{D}, \frac{0.16}{D}  \right\}, \; \; 1 \leq j \leq n.
$$
Set
$$
B= \max \{ |b_1|, |b_2|, \ldots, |b_n| \},
$$
$$
B_0=\max \left\{ 1, \max \{ |b_j| \log A_j/\log A_n : 1 \leq j \leq n \} \right\}
$$
and
$$
C= \max \{ B_0, nB \pi/(D \log A_n) \}.
$$
Then we have
$$
\log \left| \alpha_1^{b_1} \alpha_2^{b_2} \cdots \alpha_n^{b_n} -1 \right| > -3 \times 30^{n+4} (n+1)^{5.5} D^{n+2} \log (eD)
\log A_1 \cdots \log A_n \log (eC).
$$
\end{thm}
We apply this with $\alpha_n=p$, $A_n = p$ and $|b_n|=\gamma$, where
$n \leq \pi (P_0)+1$, and $\alpha_1, \ldots, \alpha_{n-1}$ are primes bounded above by $P_0$ (so that $D=1$). We further have
$$
B_0 = \frac{\kappa \log \ell}{\log p} = \theta \gamma
\; \mbox{ and } \; B \leq  \frac{\kappa \log \ell}{\log 2},
$$
so that $C \leq n \pi  \theta \gamma/\log 2$. Applying Theorem \ref{Mat}, it follows from (\ref{drive}) that there exists  constants $c_1=c_1(P_0)$ and  $c_2=c_2(P_0)$ such that 
$$
\theta \gamma \log p \leq c_1 \log p \log (c_2 \theta \gamma),
$$
whence
$$
\theta \gamma  \leq c_1  \log (c_2 \theta \gamma).
$$
It follows that $\theta \gamma$ and hence, since $\theta \geq 1$, $\gamma$ is bounded above, say $\gamma \leq \gamma_0 = \gamma_0 (P_0)$.

For each $\gamma$, $1 \leq \gamma \leq \gamma_0$, we will argue as follows.
Take $P_1, P_2, \ldots, P_s$ to be the distinct prime divisors of $q_2$, and $Q_1, Q_2, \ldots, Q_t$ to be the distinct prime divisors of $q_1$, so that $P_i, Q_i \leq P_0$ and $s+t \leq \pi (P_0)$. Write
$$
q_1 = Q_1^{b_1} \cdots Q_t^{b_t} \; \mbox{ and } \;
q_2=P_1^{a_1} \cdots P_s^{a_s}.
$$
Then, from Proposition \ref{Goop}, crudely,
\begin{equation} \label{morgan}
\left| \alpha - \frac{P_1^{[a_1/\gamma ]} \cdots P_s^{[a_s/\gamma ]} \cdot p }{Q_1^{[ b_1/\gamma ]} \cdots Q_t^{[ b_t/\gamma ]}}
\right| < \frac{3 \alpha}{ \gamma q_1^{1/2} q^{1/2}},
\end{equation}
where we write
$$
\alpha =  \frac{Q_1^{ \{ b_1/\gamma \} } \cdots Q_t^{ \{ b_t/\gamma \} }}{P_1^{\{ a_1/\gamma \}} \cdots P_s^{ \{ a_1/\gamma \} } },
$$
and $[x]$ and $\{ x \}$ denote the greatest integer and fractional part of $x$. From (\ref{good}), we have
$q_1 \equiv 3 \mod{p}$, so that $q_1-q_2p^\gamma = 3+mp$ for an integer $m$. If $m=0$, then
$$
2 q_0 p^\alpha = q_2 p^\gamma,
$$
so that $\alpha=\gamma$ and $q_2=2q_0$, whence
$(q,r_1,r_2) = \left( \frac{3^t-3}{2}, 3^t-3,3^t \right)$. As noted in Section \ref{Sec2}, in this case equation (\ref{roots1}) implies that
$2q^2 = (q-1) (n-2)$,
whereby $t=2$ and $q=3$, a contradiction.

We may therefore assume that $m \neq 0$, so that, from Proposition \ref{Goop}, 
$$
p-3 \leq \left| q_1-q_2p^\gamma \right| < \left( 6q_2p^\gamma/q \right)^{1/2}.
$$
It follows, from $p > 1000$, that
$$
q_2 > \frac{1}{7} p^{2-\gamma} q > \frac{1}{7} p^{3-\gamma}.
$$
Since
$$
Q_1^{[ b_1/\gamma ]} \cdots Q_t^{[ b_t/\gamma ]} \leq q_1^{1/\gamma},
$$
we can thus apply Theorem \ref{Riddy} with suitable absolute constant $c$, $\nu=0$, and $\mu = 1/3$ (if $\gamma=1$), $\mu = 2/3$ (if $\gamma=2$), or $\mu=1$ (if $\gamma \geq 3$), to conclude that inequality (\ref{morgan}) has at most finitely many solutions, ranging over the 
$$
\gamma^{s+t} \leq \gamma^{\pi (P_0)}
$$
possible choices for $\alpha$. This completes the proof of Theorem \ref{thm-main!}.


\end{document}